\documentclass[11 pt]{amsart}
\usepackage[english]{babel}
\usepackage{amsthm}
\usepackage{graphicx}
\usepackage{amsfonts}
\usepackage{eufrak}
\usepackage{amsopn}
\usepackage{amsmath, amsfonts, amssymb, amscd, amsthm}
\usepackage{yhmath}
\usepackage{setspace}
\usepackage{mathtools}
\usepackage[all]{xy}
\newtheorem{theorem}{Theorem}[section]

\newtheorem{lemma}[theorem]{Lemma}

\theoremstyle{definition}

\theoremstyle{remark}

\theoremstyle{example}

\theoremstyle{note}

\numberwithin{equation}{section}

\DeclareMathOperator{\Sp}{Sp}

\begin{document}
\title{Some remarks on the symplectic group $\Sp(2g,\mathbb{Z})$.}
\author{Kumar Balasubramanian; Ganesh Ji Omar}
\address{Kumar Balasubramanian, Department of Mathematics,
Indian Institute of Science Education and Research Bhopal,
Bhopal, Madhya Pradesh, India.}
\email{bkumar@iiserb.ac.in}
\address{Ganesh Ji Omar, BS-MS (2011-2016) batch,
Indian Institute of Science Education and Research Bhopal,
Bhopal, Madhya Pradesh, India.}
\email{ganesh@iiserb.ac.in}
\thispagestyle{empty}

\maketitle
\vspace{0.3 cm}
\section*{Abstract}

Let $G=\Sp(2g,\mathbb{Z})$ be the symplectic group over the integers. Given $m\in \mathbb{N}$, it is natural to ask if there exists a non-trivial matrix $A\in G$ such that $A^{m}=I$, where $I$ is the identity matrix in $G$. In this paper, we determine the possible values of $m\in \mathbb{N}$ for which the above problem has a solution. We also show that there is an upper bound on the maximal order of an element in $G$. As an illustration, we apply our results to the group $\Sp(4,\mathbb{Z})$ and determine the possible orders of elements in it. Finally, we use a presentation of $\Sp(4,\mathbb{Z})$ to identify some finite order elements and do explicit computations using the presentation to verify their orders.

\section{Introduction}

Given a group $G$ and a positive integer $m\in N$, it is natural to ask if there exists $k\neq e\in G$ such that $k^{m}=e$ where $e$ is the identity element in $G$. In this paper, we address this question in the case of the symplectic group $\Sp(2g, \mathbb{Z})$. \\

One of the principal reasons for focussing on the group $\Sp(2g, \mathbb{Z})$ is the following. It plays an important role in geometry and number theory and comes up in many interesting situations. For example, in geometry it plays a significant role in the study of certain type of surfaces due to its connections with the mapping class group. It also comes up in number theory in the study of Siegel modular forms. \\

Before we state the main results, we first recall the definition of $\Sp(2g,\mathbb{Z})$ and fix some notation. \\

The group $\Sp(2g,\mathbb{Z})$ is the group of all $2g \times 2g$ matrices with integral entries satisfying
\[A^{\top}JA=I\]
where $A^{\top}$ is the transpose of the matrix $A$ and $J=\begin{pmatrix}0_{g} & I_{g} \\ -I_{g} & 0_{g}\end{pmatrix}$.\\

Throughout we write $m=p_{1}^{\alpha_{1}}\dots p_{k}^{\alpha_{k}}$, where $p_{i}$ is a prime and $\alpha_{i}>0$ for all $i\in \{1,2, \ldots, k\}$. We also assume that the primes $p_{i}$ are such that $p_{i}<p_{i+1}$ for $1\leq i < k+1$. Also for $A\in G$ we let $o(A)$ denote the order of $A$.\\

We now state the main results of this paper.

\begin{theorem}\label{main1} Let $\mathcal{A}=\{m\in \mathbb{N} \mid p_{i}\leq 2g+1 \text{ for some } i\}$. For $A\in G$, we have $A^{m}=I$ if and only if $m\in \mathcal{A}$.
\end{theorem}
\vspace{0.3 cm}

\begin{theorem}\label{main2} Let $A \in G$ be such that $o(A)=m.$ Then $m\leq max\{30, M\}$ where $M=max\{2(2g)^{\frac{g}{\alpha}}, (2g)^{\frac{(g+1)}{\alpha}} \}$ with $\alpha=\frac{log2}{log3}$.
\end{theorem}
\vspace{0.3 cm}
The paper is organized as follows. In section~\ref{preliminaries}, we recall some important results that we need in the sequel. Section~\ref{main theorem}, contains the proofs of the main results (Theorem~\ref{main1}, Theorem~\ref{main2} ) of this paper. In section~\ref{computations}, we explicitly identify some finite order elements in $G$ when $g=2$.

\section{Preliminaries}\label{preliminaries}

Throughout this section we let $\phi$ denote the Euler's phi function. For the sake of completeness, we recall the definition of $\phi$ and record a few properties we need. For a detailed account of the properties of the Euler's $\phi$ function, we refer the reader to \cite{Burton}.\\

For $n\in \mathbb{N}$, $\phi(n)$ is defined to be the number of positive integers less than or equal to $n$ and relatively prime to $n$. The function $\phi$ is a multiplicative function. i.e., for $m,n\in \mathbb{N}$ which are relatively prime, we have $\phi(mn)=\phi(m)\phi(n)$. Using the fact that every positive integer $n>1$, can be expressed in a unique way as \[n=p_{1}^{\alpha_{1}}p_{2}^{\alpha_{2}}\dots p_{k}^{\alpha_{k}}=\prod_{i=1}^{k}p_{i}^{\alpha_{i}}\]
where $p_{1} < p_{2} < \dots < p_{k}$ are primes and $\alpha_{i}$'s are positive integers and the fact that $\phi$ is multiplicative, it is clear that $\phi(n)= \phi(p_{1}^{\alpha_{1}})\phi(p_{2}^{\alpha_{2}})\dots \phi(p_{k}^{\alpha_{k}})$. It is therefore useful to know the value of $\phi$ for prime powers. It is easy to see that \[\phi(p^{\alpha})= p^{\alpha -1}(p-1),\] where $p$ is a prime and $\alpha$ is a positive integer. \\

We now state the results we shall use to prove the main theorems in this paper. We refer the reader \cite{Sha} and \cite{Bur} for a more detailed account of these results.

\begin{theorem}[Shapiro]\label{shapiro} Let $m\in \mathbb{N}$ be such that $m\not \in \{1, 2, 3, 4, 6, 10, 12, 18, 30\}$ and $\alpha= \frac{log2}{log3}$. Then $\phi(m)>m^{\alpha}$.
\end{theorem}

\begin{theorem}[B\"{u}rgisser]\label{burgisser} Let $\displaystyle m= p_{1}^{\alpha_{1}}\dots p_{k}^{\alpha_{k}}$, where the primes $p_{i}$ satisfy $p_{i}<p_{i+1}$ for $1\leq i < k$ and where $\alpha_{i}\geq 1$ for $1\leq i \leq k$. There exists a matrix $A\in \Sp(2g, \mathbb{Z})$ of order $m$ if and only if\\

\begin{enumerate}
\item[a)] $\displaystyle \sum_{i=2}^{k}\phi(p_{i}^{\alpha_{i}})\leq 2g$, if $m\equiv 2(mod 4)$.
\item[b)] $\displaystyle \sum_{i=1}^{k}\phi(p_{i}^{\alpha_{i}})\leq 2g$, if $m\not \equiv 2(mod 4)$.
\end{enumerate}
\end{theorem}

\section{Main Results}\label{main theorem}
Throughout this section we take $m\in \mathbb{N}$ to be as in Theorem~\ref{burgisser}. We let $\mathcal{A}=\{m\in \mathbb{N} \mid p_{i}\leq 2g+1 \text{ for some } i\}$ and $\mathcal{B}=\mathbb{N}\setminus A$. Before we prove the main theorem, we record a lemma we need. 

\begin{lemma}\label{possible primes} Let $A\in G$ such that $o(A)=m$. Then $p_{i}\leq 2g+1, \forall i\in \{1, 2, \ldots ,k\}$.
\end{lemma}

\begin{proof} Suppose $p_{i}>2g+1$ for some $i\in \{1, 2, \ldots, k\}$. This would imply that $\phi(p_{i}^{\alpha_{i}})=p_{i}^{\alpha_{i}-1}(p_{i}-1)>2g$ and we get a contradiction to Theorem~\ref{burgisser}. 
\end{proof}

\begin{theorem}\label{powers in G} Let $m\in \mathbb{N}$. Then $A^{m}=I$ if and only if $m\in \mathcal{A}$.
\end{theorem}

\begin{proof} Suppose $m\in \mathcal{A}$. Choose $p_{i}$ such that $p_{i}\leq 2g+1$. Clearly, $\phi(p_{i})\leq 2g$ and it follows from theorem~\ref{burgisser} that we have $A\in G$ such that $o(A)=p_{i}$. Let $n=p_{1}^{\alpha_{1}}\dots p_{i}^{\alpha_{i}-1}\dots p_{k}^{\alpha_{k}}$. Now $A^{m}=(A^{p_{i}})^{n}=I$. \\

Suppose there exists $A\in G$ such that $A^{m}=I$. We show that $m\not \in \mathcal{B}$. Since $A^{m}=I$, it follows that $o(A)|m$. Let $o(A)=n=q_{1}^{r_{1}}\dots q_{\ell}^{r_{\ell}}$. By lemma~\ref{possible primes}, we know that each $q_{i}\leq 2g+1$. The result now follows from the following simple observation. Since $q_{i}|n$ and $n|m$, we have $q_{i}|m$ for each $i\in \{1, 2, \ldots, \ell\}$. This is indeed not possible if $m\in \mathcal{B}$.
\end{proof}

\subsection{An upper bound for the order}
We show that the maximal order in $G$ is always bounded. First, we introduce some notation.\\

Let $m\in \mathbb{N}$, be as in Theorem~\ref{shapiro} and let $n=p_{2}^{\alpha_{2}}\dots p_{k}^{\alpha_{k}}$. Suppose that there exists $A\in G$ such that $o(A)=m$. Consider the following sums: \[\displaystyle S_{1}= \sum_{i=1}^{k}\phi(p_{i}^{\alpha_{i}}) \quad
\text{ if } m\not \equiv 2(mod\, 4) \text{\quad  and } \]
\[\displaystyle S_{2}= \sum_{i=2}^{k}\phi(p_{i}^{\alpha_{i}}) \quad
\text{ if } m\equiv 2(mod\, 4).\]
\vspace{0.3 cm}\\
In the following lemmas, We show that $S_{1}$ and $S_{2}$ are bounded below.

\begin{lemma}\label{estimate for S1} $\displaystyle S_{1}=\sum_{i=1}^{k}\phi(p_{i}^{\alpha_{i}}) > m^{\frac{\alpha}{(g+1)}}$.
\end{lemma}

\begin{proof} We know that $p_{i}\leq 2g+1, 1\leq i\leq k$. From this it follows that $k\leq g+1$. Now consider the sum $S_{1}$. We have
\begin{align*}
\phi(p_{1}^{\alpha_{1}})+ \dots + \phi(p_{k}^{\alpha_{k}}) & \geq k(\phi(p_{1}^{\alpha_{1}})\dots\phi(p_{k}^{\alpha_{k}}))^{\frac{1}{k}}\\
&= k(\phi(m))^{\frac{1}{k}}\\
&> m^{\frac{\alpha}{k}}\\
&\geq m^{\frac{\alpha}{(g+1)}}.
\end{align*}
\end{proof}

\begin{lemma}\label{estimate for S2} $\displaystyle S_{2}=\sum_{i=2}^{k}\phi(p_{i}^{\alpha_{i}}) > n^{\frac{\alpha}{g}}$.
\end{lemma}

\begin{proof} Since $m\equiv 2(mod\, 4)$, it follows that $m=2n=2(p_{2}^{\alpha_{2}}\dots p_{k}^{\alpha_{k}})$. Clearly, $n\not \in \{1, 2, 3, 4, 6, 10, 12, 18, 30\}$ and the inequality in theorem~\ref{shapiro} applies. Applying a similar argument as in lemma~\ref{estimate for S1} to $n$ gives us the desired lower bound for $S_{2}$.
\end{proof}

\begin{theorem}\label{bound on the order} Let $A \in G$ be such that $o(A)=m.$ Then $m\leq max\{30, M\}$ where $M=max\{2(2g)^{\frac{g}{\alpha}}, (2g)^{\frac{(g+1)}{\alpha}} \}$ with $\alpha=\frac{log2}{log3}$.\\
\end{theorem}

\begin{proof} Suppose that $m \not \equiv 2(mod\, 4)$. By lemma~\ref{estimate for S1}, we have $S_{1}> m^{\frac{\alpha}{(g+1)}}$. If $m>(2g)^{\frac{(g+1)}{\alpha}}$, then we have $S_{1}>2g$. This is clearly not possible. Thus it follows that $m\leq (2g)^{\frac{(g+1)}{\alpha}}$. \\

Similarly, we see that if $m\equiv 2(mod \, 4)$, then lemma~\ref{estimate for S2} applies and we have $S_{2}>n^{\frac{\alpha}{g}}$. If $m>2(2g)^{\frac{g}{\alpha}}$, then $S_{2}>2g$. As this is not possible, it follows that $m\leq 2(2g)^{\frac{\alpha}{g}}$. \\

Taking $M=max\{2(2g)^{\frac{g}{\alpha}}, (2g)^{\frac{(g+1)}{\alpha}} \}$, we obtain $m\leq max\{30, M\}$.
\end{proof}

\section{Finite order elements in $\Sp(4, \mathbb{Z})$}\label{computations}

Using Lemma~\ref{possible primes}, Theorem~\ref{bound on the order} and Theorem~\ref{burgisser}, it is easy to see that the possible orders of elements in $\Sp(4,\mathbb{Z})$ are precisely $2, 3, 4, 5, 6, 8, 10$ and $12$. Since this is computational, we leave the details to the reader. In this section, we explicitly identify matrices in $\Sp(4, \mathbb{Z})$ of these orders. The main tool we use is Bender's presentation of $\Sp(4, \mathbb{Z})$. Throughout this section we use the same notation as in \cite{Ben} in all our computations. \\

In \cite{Ben}, Bender gives a presentation of $\Sp(4,\mathbb{Z})$ using two generators and eight defining relations. We recall his result below.

\begin{theorem}[Bender]\label{bender} The group $\Sp(4, \mathbb{Z})$ is generated by the two elements
\[K= \begin{bmatrix}1 & 0 & 0 & 0\\ 1 & -1 & 0 & 0 \\ 0 & 0 & 1 & 1 \\ 0 & 0 & 0 & -1 \end{bmatrix}, \quad L= \begin{bmatrix}0 & 0 & -1 & 0 \\ 0 & 0 & 0 & -1\\ 1 & 0 & 1 & 0\\ 0 & 1 & 0 & 0  \end{bmatrix}\]
subject to the following eight relations:\\

\begin{enumerate} \item[a)] $K^{2}=I$,
\item[b)] $L^{12}=I$,
\item[c)] $(KL^{7}KL^{5}K)L=L(KL^{5}KL^{7}K)$,
\item[d)] $(L^{2}KL^{4})(KL^{5}KL^{7}K)= (KL^{5}KL^{7}K)(L^{2}KL^{4}),$
\item[e)] $(L^{3}KL^{3})(KL^{5}KL^{7}K)= (KL^{5}KL^{7}K)(L^{3}KL^{3}),$
\item[f)] $(L^{2}(KL^{5}KL^{7}K))^{2}= ((KL^{5}KL^{7}K)L^{2})^{2},$
\item[g)] $L(L^{6}(KL^{5}KL^{7}K))^{2}= (L^{6}(KL^{5}KL^{7}K))^{2}L,$
\item[h)] $(KL^{5})^{5}= (L^{6}(KL^{5}KL^{7}K))^{2}$.
\end{enumerate}
\end{theorem}

Before we proceed further, we quickly recall some notation from \cite{Ben} that we frequently use. We write the exponent $m$ for the word $L^{m}, (m\in \mathbb{N})$ and $H$ for the word $KL^5KL^7K$. For example, in our new notation, the word $H=KL^5KL^7K$ will be written as $H=K5K7K$. We also note that $H^2=I$ and let $w_{\alpha}=H6$, $w_{\beta}=9H6H$ and $x_{\alpha}=5K1$. \\

It is clear from the presentation that we have elements of orders $2, 3, 4, 6 $ and $12$ ($o(K)=2$, $o(L^{4})=3$, $o(L^{3})=4$, $o(L^{2})=6$ and $o(L)=12$).\\

Consider the word $K5$. We first show that $(K5)^{10}=I$. Using $(h)$ it follows that $(K5)^5 =(6H)^2$. It is enough to show that $(6H)^{4}=I$. Indeed,
\begin{align*}
(6H)^{4} &= (6H)^{2}(6H)^{2}\\
& \overset{(i)}= (H6)^{2}(H6)^{2}\\
&= H6(H6)^{2}H6\\
&\overset{(i)}= H6(6H)^{2}H6\\
&= H\underline{66}H6\underline{HH}6\\
&= I.
\end{align*}

Since $(K5)^{10}=I$, it follows that $o((K5)^2)$ is either $1$ or $5$. We will prove that $o((K5)^2)=5$ by showing that $(K5)^2\neq I$. Suppose $(K5)^2=I$. A simple computation shows that $K5=5K$ and we get a contradiction. Indeed,
\begin{align*}
K5 &\overset{(h)}= (6H)^{2}\\
& \overset{(i)}= (HK)^{3}\\
K5K &= (HK)^{3}K\\
&\overset{(i)}= K(HK)^{3}\\
&= K(K5)\\
&= 5\\
K5 &= 5K
\end{align*}

The result follows from the computation below.
\begin{equation}\label{eqn 1}
I= (K5)^{2}= (K5)(K5)= (K5)(5K)= K10K= 10.
\end{equation}

Since $(K5)^{10}=I$, it follows that $o(K5)$ is either $1, 2, 5$ or $10$. We will show that $o(K5)=10$. It is enough to show that $(K5)^{5}\neq I$ (we know that $K5\neq I$ and $(K5)^{2}\neq I$). Suppose $(K5)^{5}=I$. Using $(h)$ and $(i)$ it follows that \[(K5)^{5}=(6H)^{2}=(HK)^{3}=w_{\alpha}^{2}=I.\] Since $w_{\alpha}^{2}=1$, we see that
\begin{equation}\label{eqn 2}
6H=H6.
\end{equation}

Using ~\eqref{eqn 2}, we have
\begin{equation}\label{eqn 3}
w_{\beta}=9H\underline{6H}=9HH6=3.
\end{equation}

Before we proceed further, we do a computation which is essential. To be more precise, we show that
\begin{equation}\label{eqn 4}
3=w_{\beta}= 8K7KHK.
\end{equation}
Indeed,
\begin{align*}
3 &= w_{\beta} \\
&= \underline{9H6}H\\
&\overset{(k)}= (8K7\underline{K5K7})(\underline{K5K7}K)\\
&= 8K7(K5K7)^{2}K\\
&= 8K7(HK)^{2}K\\
&= 8K7KHK \quad [\text{since $(HK)^{3}=1$ by assumption}]\\
\end{align*}

From ~\eqref{eqn 4}, it follows that
\begin{equation}\label{eqn 5}
w_{\beta}^{2}=6=11K7KHK.
\end{equation}

We will use ~\eqref{eqn 5} to show that $H=K$ and ultimately $K5=5K$ giving us a contradiction. Before we continue we make the following observation. We show that \begin{equation}\label{eqn 6}
(KH)^{2}K=KHK.
\end{equation}

This follows from $(i)$ and the fact that $(HK)^{3}=I$. Indeed,
\begin{displaymath}
(KH)(KH)^{2}K = (KH)^{3}K \overset{(i)}= K(HK)^{3}=K.
\end{displaymath}

Consider ~\eqref{eqn 5}. We have
\begin{align*}
6 &= \underline{11K7}KHK\\
& \overset{(2)}= H6KHKHK\\
&= \underline{H6}\underline{(KH)^{2}K}\\
& \overset{~\eqref{eqn 6}}= 6HKHK\\
&= 6(HK)^{2}\\
&= 6KH
\end{align*}

The above computation shows that $KH=1$ and $H=K$. Since we have $H=K$, it follows that $K5K7=1$ and $K5=5K$. The result follows from ~\eqref{eqn 1}.\\

Consider the word $9H$. Clearly, $9H\neq I$. We will show that $(9H)^4 \neq I$ and $(9H)^{8}=I$. \\

Using $(9)$ in \cite{Ben}, we have
\begin{align*}
(9H)^{4} &= (w_{\alpha}w_{\beta})^{4}\\
&= (w_{\alpha}w_{\beta})(w_{\beta}w_{\alpha})^{2}(w_{\alpha}w_{\beta})\\
&= w_{\alpha}w_{\beta}^{2}w_{\alpha}w_{\beta}w_{\alpha}^{2}w_{\beta}\\
& \overset{(7)}= w_{\beta}^{3}\underline{w_{\alpha}^{2}w_{\beta}}\\
& \overset{(8)}= w_{\beta}^{3}w_{\beta}w_{\alpha}^{2}\\
& \overset{(10)}= w_{\alpha}^{2}.
\end{align*}

The result now follows from a simple observation. We have
\[(9H)^{4}=w_{\alpha}^{2}=(6H)^{2}\overset{(i)}= (H6)^{2} \overset{(h)}= (K5)^{5}\neq I\]
and \[(9H)^{8}=(K5)^{10}=I.\]

\section*{Acknowledgements}\label{gap}

We would like to thank Dr. Gururaja Upadhya for showing interest in this work. We would also like to thank him for his suggestions in improving the presentation of this paper.

\bibliographystyle{amsplain}
\bibliography{ref}
\end{document}